\documentclass[12 pt]{article}
\usepackage{amsmath,amsfonts,amsthm,amssymb,array,mathrsfs,latexsym,paralist, xypic, color,tabu}
\usepackage{tikz, soul}
\usepackage{lipsum}

\usepackage{dsfont}

\usetikzlibrary{arrows,automata}

\tikzstyle{V}=[fill=black,circle,scale=0.2, outer sep = 4pt]

\newtheorem{thm}{Theorem}[section]
\newtheorem{prop}[thm]{Proposition}

\theoremstyle{remark}
\newtheorem{rmk}[thm]{Remark}

\theoremstyle{definition}
\newtheorem{defn}[thm]{Definition}

\newcommand{\bi}{\begin{itemize}}
\newcommand{\ei}{\end{itemize}}
\newcommand{\be}{\begin{enumerate}}
\newcommand{\ee}{\end{enumerate}}

\newcommand{\C}{\mathbb{C}}

\renewcommand{\H}{\mathcal{H}}

\newcommand{\R}{\mathbb{R}}
\newcommand{\N}{\mathbb{N}}

\providecommand{\keywords}[1]{{\textit{Key words and phrases:}} #1}
\providecommand{\classification}[1]{{\textit{2010 Mathematics Subject Classification:}} #1}

\def\IoIIdimdots(#1/#2/#3,#4){\node at (#1,#4) {$.$};\node at (#2,#4) {$.$};\node at (#3,#4) {$.$};}
\def\IIoIIdimdots(#1,#2/#3/#4){\node at (#1,#2) {$.$};\node at (#1,#3) {$.$};\node at (#1,#4) {$.$};}

\def\IoIIIdimdots(#1/#2/#3,#4,#5){\node at (#1,#4,#5) {$.$};\node at (#2,#4,#5) {$.$};\node at (#3,#4,#5) {$.$};}
\def\IIoIIIdimdots(#1,#2/#3/#4,#5){\node at (#1,#2,#5) {$.$};\node at (#1,#3,#5) {$.$};\node at (#1,#4,#5) {$.$};}
\def\IIIoIIIdimdots(#1,#2,#3/#4/#5){\node at (#1,#2,#3) {$.$};\node at (#1,#2,#4) {$.$};\node at (#1,#2,#5) {$.$};}

\usepackage[margin=0.77in]{geometry}

\AtEndDocument{\bigskip{\small%
  \textsc{Carla Farsi, Judith Packer : Department of Mathematics, University of Colorado at Boulder, Boulder, Colorado, 80309-0395, USA.} \par
  \textit{E-mail address}: \texttt{carla.farsi@colorado.edu, packer@euclid.colorado.edu} \par
  \addvspace{\medskipamount}
  \textsc{Elizabeth Gillaspy : Department of Mathematical Sciences, University of Montana, 32 Campus Drive \#0864, Missoula, MT 59812-0864.}\par
  \textit{E-mail address}: \texttt{elizabeth.gillaspy@mso.umt.edu}\par
  \addvspace{\medskipamount}
  \textsc{Antoine Julien : Nord University Levanger, H\o gskoleveien 27, 7600 Levanger, Norway. }\par
  \textit{E-mail address}: \texttt{antoine.julien@nord.no}\par
   \addvspace{\medskipamount}
   \textsc{Sooran Kang : College of General Education, Chung-Ang University, 84 Heukseok-ro, Dongjak-gu, Seoul, Republic of Korea.} \par
  \textit{E-mail address}, \texttt{sooran09@cau.ac.kr}
}}

\begin{document}

\title{Spectral triples for higher-rank graph $C^*$-algebras}

\author{Carla Farsi, Elizabeth Gillaspy, Antoine Julien, Sooran Kang, and Judith Packer}

\date{\today}

\maketitle

\begin{abstract}

In this note, we present a new way to associate a spectral triple to the noncommutative $C^*$-algebra $C^*(\Lambda)$ of a strongly connected finite higher-rank graph $\Lambda$.  
We generalize a spectral triple of Consani and Marcolli from Cuntz-Krieger algebras to higher-rank graph $C^*$-algebras $C^*(\Lambda)$, and we prove that these spectral triples are intimately connected to the wavelet decomposition of the infinite path space of $\Lambda$ which was introduced by Farsi, Gillaspy, Kang, and Packer in 2015. In particular, we prove that the wavelet decomposition of Farsi et al.~describes the eigenspaces of the  Dirac operator of this spectral triple.
   \end{abstract}

\classification{46L05, 46L87, 58J42.}


\tableofcontents

\section{Introduction}

Inspired by constructions from Arakelov geometry and Archimedean cohomology,
Consani and Marcolli 
develop in \cite{consani-marcolli}   spectral triples associated to certain Cuntz--Krieger algebras.
In this note, we expand the applicability of these spectral triples by generalizing the construction of \cite{consani-marcolli} to the setting of higher-rank graphs.  We also  establish the compatibility of these spectral triples with the representations and wavelets for higher-rank graphs which were developed in \cite{FGKP}. Indeed, both spectral triples and wavelets are algebraic structures which encode geometrical information, so it is natural  to ask about the relationship between   wavelets  and spectral triples. 

Our   earlier paper \cite{FGJKP2} was the first to establish  a connection between wavelets and spectral triples in the setting of higher-rank graphs $\Lambda$.  In that paper, we linked the   representations of $C^*(\Lambda)$ from  \cite{FGKP}, and their associated wavelets, to the eigenspaces of the Laplace--Beltrami operators which arise from the spectral triples of Pearson and Bellissard \cite{pearson-bellissard}.  The present article establishes that the wavelets from \cite{FGKP} can also be identified with the eigenspaces of the Dirac operator of a Consani--Marcolli type spectral triple for $C^*(\Lambda)$.

 Higher-rank graphs (also called $k$-graphs)  were introduced by Kumjian and Pask in \cite{kp} to  provide a combinatorial model to the higher-dimensional Cuntz-Krieger algebras given by Robertson and Steger in \cite{robertson-steger}. The $C^*$-algebras $C^*(\Lambda)$ of $k$-graphs $\Lambda$ have been studied by many authors and provided concrete, computable examples of many classifiable $C^*$-algebras. The graphical character of $k$-graphs has also facilitated the analysis of structural properties of $C^*(\Lambda)$,  such as  simplicity and ideal structure \cite{rsy2, robertson-sims, davidson-yang-periodicity, kang-pask, ckss}, quasidiagonality \cite{clark-huef-sims} and KMS states \cite{aHLRS, aHLRS1, aHKR}.

However, the analysis of the noncommutative geometry of $C^*(\Lambda)$ is in its infancy. 
Although Pask, Rennie, and Sims establish in \cite{pask-rennie-sims-manifold} that higher-rank graph $C^*$-algebras often provide tractable  examples of noncommutative manifolds, the current literature contains only one class of (semifinite) spectral triples  for $C^*(\Lambda)$, namely those studied in \cite{pask-rennie-sims}. In the Pearson--Bellissard spectral triples $(\mathcal A, \H, D)$ which were associated to higher-rank graphs in \cite{FGJKP2}, the algebra $\mathcal A = C_{Lip}(\Lambda^\infty)$ is  commutative.   Thus,  the spectral triples for the noncommutative $C^*$-algebra $C^*(\Lambda)$, which we construct in Theorem \ref{thm-Consani-Marcolli-spectral-triples-k-graphs}  below, constitute an important step forward in our understanding of the noncommutative geometry of $C^*(\Lambda)$, in particular because of the link we establish between these spectral triples and wavelet theory for $C^*(\Lambda)$.

Wavelets for higher-rank graphs $\Lambda$ were introduced by four of the authors of the current paper in \cite{FGKP}, building on work of Marcolli and Paolucci \cite{marcolli-paolucci} for Cuntz--Krieger algebras, which in turn was inspired by the wavelets for fractal spaces developed by  Jonsson \cite{jonsson} and Strichartz \cite{strichartz}.  In all of these settings, the wavelets give an orthogonal decomposition of $L^2(X, \mu)$ for a fractal space $X$, which arises from applying dilation and translation operators to a finite family of ``mother wavelets'' $f_i \in L^2(X, \mu)$. The dilation and translation operators are determined by the underlying geometry.  In Jonsson and Strichartz' work,  the self-similar structure of the fractal space $X$ dictates the dilation and translation operators, while in the higher-rank graph case, the dilation and translation operators arise from the graph structure.  (See Section~\ref{sec:wavelets} for more details.)

To further our understanding of the noncommutative geometry of $C^*(\Lambda)$, 
 we  construct in Theorem~\ref{thm-Consani-Marcolli-spectral-triples-k-graphs} a spectral triple $(\mathcal{A}_\Lambda, L^2(\Lambda^\infty, M), D)$, where $\mathcal A_\Lambda$ is a dense (noncommutative) subalgebra of $C^*(\Lambda)$. This spectral triple was inspired by the spectral triples for Cuntz--Krieger algebras constructed in \cite{consani-marcolli}, and offers a very different perspective on the noncommutative geometry of $C^*(\Lambda)$ than the spectral triples of \cite{pask-rennie-sims}.   Theorem \ref{thm:CM-Dirac-wavelets} then establishes our link between spectral triples and wavelets for higher-rank graphs by showing that the eigenspaces of the Dirac operator $D$ of this spectral triple agree with the wavelet decomposition of \cite{FGKP}.

\subsection*{Acknowledgments} E.G.~was partially supported by the SFB 878 ``Groups, Geometry, and Actions'' of the Westf\"alische-Wilhelms-Universit\"at M\"unster. C.F.~and J.P.~were  partially supported by two individual  grants from the Simons Foundation (C.F. \#523991; J.P. \#316981).
S.K.~was supported by Basic Science Research Program through the National Research Foundation of Korea (NRF) funded by the Ministry of Education (\#2017R1D1A1B03034697).

\section{Background material}\label{sec:background}

We begin by detailing some foundational material needed for our results, and in particular  reviewing the definition of a higher-rank graph $\Lambda$, the definition of its $C^*$-algebra $C^*(\Lambda)$, and associated wavelets.

\subsection{Higher-rank graphs and their $C^*$-algebras}
\label{sec:kgraph}
Throughout this paper, we will view 
$\N: =\{0,1,2,\dots\}$ as a monoid under addition, or as a category.  In this interpretation, the natural numbers are the morphisms in $\N$.  Thus, for consistency with the standard notation $n \in \N$, we will write 
\[ \lambda \in \Lambda \]
to indicate that $\lambda $ is a morphism in the category $\Lambda$.

\begin{defn}
\label{def:k-graph}
A {\em higher-rank graph} or {\em $k$-graph} by definition is a countable small category $\Lambda$ with a degree functor $d:\Lambda\to \N^k$ satisfying the {\em factorization property}: for any morphism $\lambda\in\Lambda$ and any $m, n \in \N^k$ such that  $d(\lambda)=m+n \in \N^k$,  there exist unique morphisms $\mu,\nu\in\Lambda$ such that $\lambda=\mu\nu$ and $d(\mu)=m$, $d(\nu)=n$. 

We often think of $k$-graphs as a generalization of directed graphs, so we call objects $v \in\Lambda^0$ ``vertices'' and  morphisms $\lambda\in\Lambda$ are called ``paths.'' We write $r,s:\Lambda\to \Lambda^0$ for the range and source maps and $v \Lambda w=\{\lambda\in \Lambda: r(\lambda)=v, s(\lambda)=w\}$. Similarly, for any $n\in \N^k$, we write $v\Lambda^n=\{\lambda\in\Lambda: r(\lambda)=v, d(\lambda)=n\}$.
\end{defn}

 For $m,n\in\N^k$, we denote by $m\vee n$ the coordinatewise maximum of $m$ and $n$. Given  $\lambda,\eta\in \Lambda$, we write
\begin{equation*}%\label{eq:lambda_min}
\Lambda^{\operatorname{min}}(\lambda,\eta):=\{(\alpha,\beta)\in\Lambda\times\Lambda\,:\, \lambda\alpha=\eta\beta,\; d(\lambda\alpha)=d(\lambda)\vee d(\eta)\}.
\end{equation*}

We say that a $k$-graph $\Lambda$ is \emph{finite} if $\Lambda^n$ is a finite set for all $n\in\N^k$ and say that $\Lambda$  \emph{has no sources} or \emph{is source-free} if $v\Lambda^n\ne \emptyset$ for all $v\in\Lambda^0$ and $n\in\N^k$. It is well known that this is equivalent to the condition that $v\Lambda^{e_i}\ne \emptyset$ for all $v\in \Lambda$ and all basis vectors $e_i$ of $\N^k$. Also we say that a $k$-graph is \emph{strongly connected} if, for all $v,w\in\Lambda^0$, $v\Lambda w\ne \emptyset$.

%Now we introduce $C^*(\Lambda)$, the $C^*$-algebra associated to a $k$-graph $\Lambda$, as follows. 

\begin{defn}\cite{kp}
If $\Lambda$ is a finite $k$-graph with no sources, write $C^*(\Lambda)$ for the universal $C^*$-algebra generated by partial isometries $\{s_\lambda\}_{\lambda \in \Lambda}$ satisfying the Cuntz--Krieger conditions:
\begin{itemize}
\item[(CK1)] $\{ s_v: v \in \Lambda^0\}$ is a family of mutually orthogonal projections;
\item[(CK2)] Whenever $s(\lambda) = r(\eta)$ we have $s_\lambda s_\eta = s_{\lambda \eta}$;
\item[(CK3)] For any $\lambda \in \Lambda, \ s_\lambda^* s_\lambda = s_{s(\lambda)}$;
\item[(CK4)] For all $v \in \Lambda^0$ and all $n \in \N^k$, $\sum_{\lambda \in v\Lambda^n} s_\lambda s_\lambda^* = s_v$.
\end{itemize}
\end{defn}

Condition (CK4) implies that for any $\lambda, \eta \in \Lambda$ we have 
\[ s_\lambda^* s_\eta = \sum_{(\alpha, \beta) \in \Lambda^{min}(\lambda, \eta)} s_\alpha s_\beta^*,\]
where we interpret empty sums as zero.  Consequently, $C^*(\Lambda) = \overline{\text{span}}\{ s_\lambda s_\eta^*: \lambda, \eta \in \Lambda\}$. 
\begin{defn}
\label{def:dense-subalg}
Let  $\mathcal{A}_\Lambda$ denote  the dense $*$-subalgebra of $C^*(\Lambda)$ spanned by $\{ s_\lambda s_\eta^*\}_{\lambda, \eta \in \Lambda}$.
\end{defn}

An important example of a $k$-graph is the category $\Omega_k$, where 
\[ \text{Obj}(\Omega_k) = \N^k, \qquad \text{Mor}(\Omega_k) = \{ (p, q) \in \N^k: p \leq q\}.\]
The  range and source maps $r,s$ in $\Omega_k$ are given by $r(p,q)=p$, $s(p,q)=q$, and the degree map $d: \Omega_k \to \N^k$ is given by 
\[ d(p, q) = q-p.\]

\begin{defn}
\label{def:infinite-path}
An {\em infinite path} in a $k$-graph $\Lambda$ is a degree preserving functor $x:\Omega_k\to \Lambda$.  We write $\Lambda^\infty$ for the set of infinite paths in $\Lambda$.

Given $\lambda \in \Lambda$, we define the {\em cylinder set}
$ [\lambda] \subseteq \Lambda^\infty$ by 
\[[\lambda] := \{ x \in \Lambda^\infty: x(0, d(\lambda)) = \lambda\}\]
to be the infinite paths with initial segment $\lambda$. It is well-known (cf.~\cite{kp})  that the collection of cylinder sets $\{[\lambda]\}_{\lambda \in \Lambda}$  forms a compact open basis for a locally compact Hausdorff topology on $\Lambda^\infty$. If a $k$-graph $\Lambda$ is finite, then $\Lambda^\infty$ is compact in this topology.

 For each $m\in \N^k$, we have a shift map $\sigma^m$ on $\Lambda^\infty$ given by 
\begin{equation}\label{eq:shift-map}
\sigma^m(x)(p,q)=x(p+m, q+m).
\end{equation}
 for $x\in \Lambda^\infty$ and $(p,q)\in \Omega_k$.
In duality to the shift map $\sigma^m$, for each $\lambda \in \Lambda$ we also have  a prefixing map $\sigma_\lambda: [s(\lambda)] \to [\lambda]$ given by 
\begin{equation}\label{eq:prefix-map}
 \sigma_\lambda(x) = \lambda x = \left[ (p, q) \mapsto \begin{cases} \lambda(p, q), & q \leq d(\lambda) \\
x(p-d(\lambda), q-d(\lambda)), & p \geq d(\lambda) \\
\lambda (p, d(\lambda))\, x(0, q-d(\lambda)), & p < d(\lambda) < q
\end{cases} \right]
\end{equation}
 \end{defn}

According to \cite[Proposition 8.1]{aHLRS}, for any   finite and strongly connected $k$-graph $\Lambda$, there is a unique self-similar Borel probability measure $M$ on $\Lambda^\infty$.  To describe $M$, we require more definitions.

\begin{defn}
\label{def:vertex-matrix}
For a finite $k$-graph $\Lambda$ and $1 \leq i \leq k$, the {\em vertex matrix}
$A_i \in M_{\Lambda^0}(\N)$ is 
\[ A_i(v,w) = \# (v\Lambda^{e_i} w).\]
\end{defn}
Lemma 3.1 of \cite{aHLRS} establishes that if $\Lambda$ is finite and strongly connected, then there exists a unique vector $\kappa^\Lambda \in (0, \infty)^{\Lambda^0}$, called the {\em Perron--Frobenius eigenvector} of $\Lambda$, such that 
\[ \sum_{v\in \Lambda^0} \kappa^\Lambda_v  =1 \qquad \text{ and } \qquad A_i \kappa^\Lambda = \rho_i \kappa^\Lambda \quad \forall \, 1 \leq i \leq k.\]

 The unique self-similar Borel probability measure $M$ of \cite{aHLRS} is given on cylinder sets by
\[
M([\lambda])=(\rho(\Lambda))^{-d(\lambda)}\kappa^{\Lambda}_{s(\lambda)}\quad\text{for}\;\; \lambda\in\Lambda.
\]
Here $\rho(\Lambda)=(\rho_1,\dots \rho_k)$, where $\rho_i$ denotes the spectral radius of the vertex matrix $A_i \in M_{\Lambda^0}(\N)$,
 and $(\rho(\Lambda))^n:=\rho_1^{n_1}\dots \rho_k^{n_k}$ for $n=(n_1,\dots n_k)\in \R^k$. We call the measure $M$ the \emph{Perron--Frobenius measure} on $\Lambda^\infty$.

\subsection{Wavelets on higher-rank graphs}
\label{sec:wavelets}
According to Proposition~3.4 and Theorem~3.5 of \cite{FGKP}, there is a separable representation $\pi$ of $C^*(\Lambda)$ on $L^2(\Lambda^\infty, M)$ when $\Lambda$ is a finite, strongly connected $k$-graph.  Theorem \ref{thm-Consani-Marcolli-spectral-triples-k-graphs} below  identifies a Dirac operator $D$ for which this representation gives a  spectral triple $(\mathcal{A}_\Lambda, L^2(\Lambda^\infty, M), D)$.

{Before stating Theorem \ref{thm-Consani-Marcolli-spectral-triples-k-graphs}, we review the definition of the representation $\pi$ and the associated wavelet decomposition of $L^2(\Lambda^\infty, M)$.
For $p\in \N^k$ and $\lambda\in \Lambda$, let $\sigma^p$ and $\sigma_\lambda$ be the shift  and prefixing maps on $\Lambda^\infty$ given in \eqref{eq:shift-map} and \eqref{eq:prefix-map}.
If we let $S_\lambda:=\pi(s_\lambda)$, the image of the standard generator $s_\lambda$ of $C^*(\Lambda)$ under the representation $\pi$, then \cite[Theorem 3.5]{FGKP} tells us that $S_\lambda$ is given on characteristic functions of cylinder sets by
\begin{equation}\label{eq:S_lambda}
\begin{split}
S_\lambda\chi_{[\eta]} (x) &=\chi_{[\lambda]}(x)\rho(\Lambda)^{d(\lambda)/2}\chi_{[\eta]}(\sigma^{d(\lambda)}(x))=\begin{cases} \rho(\Lambda)^{d(\lambda)/2}\quad \text{if $x=\lambda\eta y$ for some $y \in \Lambda^\infty$}\\ 
     0 \quad\quad\quad \text{otherwise}\end{cases}\\
     &= \rho(\Lambda)^{d(\lambda)/2} \chi_{[\lambda \eta]}(x).
\end{split}
\end{equation}

   Moreover, the  adjoint $S^*_\lambda$ of  $S_\lambda$  is given on characteristic functions of cylinder sets by
    \begin{equation}
\label{eq:S-lambda-star}    
    \begin{split}
    S^*_\lambda \chi_{[\eta]}(x) &=\chi_{[s(\lambda)]}(x) \rho(\Lambda)^{-d(\lambda)/2}\chi_{[\eta]}(\sigma_\lambda(x))=\begin{cases} \rho(\Lambda)^{-d(\lambda)/2}\quad\text{if $\lambda x=\eta y$ for some $y\in \Lambda^\infty$}\\ 0 \quad\quad\quad\text{otherwise}\end{cases}\\
    &= \rho(\Lambda)^{-d(\lambda)/2}\sum_{(\zeta, \xi) \in \Lambda^{min}(\lambda, \eta)} \chi_{[\zeta]}(x).
   \end{split}\end{equation}
 }

We can think of the operators $S_{\lambda}$ as combined ``scaling and translation'' operators, since they change both the size and the range of a cylinder set $[\eta]$, and are intimately tied to the geometry of the $k$-graph $\Lambda$.

This perspective enabled four of the authors of the current paper to use the representation $\pi$ to construct a wavelet decomposition of $L^2(\Lambda^\infty, M)$; we recall the details from \cite[Section 4]{FGKP}.
 For each vertex $v$ in $\Lambda$, let 
\[ D_v =  v\Lambda^{(1,\ldots, 1)} .\] 
One can show (cf.~\cite[Lemma 2.1(a)]{aHLRS}) that $D_v$ is always nonempty when $\Lambda$ is strongly connected.

Enumerate the elements of $D_v$ as $D_v = \{ \lambda_0, \ldots, \lambda_{\#(D_v) -1}\}.$
Observe that if $D_v = \{ \lambda\}$ is a 1-element set, then $[v] = [\lambda]$.  If $\#(D_v) > 1$,
then for each  $1 \leq i \leq \#(D_v) -1$, we define 
\begin{equation}\label{eq:f_iv}
f^{i,v} = \frac{1}{M[\lambda_0]} \chi_{[\lambda_0]} - \frac{1}{M[\lambda_i]} \chi_{[\lambda_i]}.
\end{equation}
One easily checks that in $L^2(\Lambda^\infty, M)$, $\langle f^{i,v} , \chi_{[w]} \rangle = 0$ for all $i$ and all vertices $v, w$, and that \[ \{ f^{i, v}: v \in \Lambda^0, 1 \leq i\leq \#(D_v)-1\}\]
is an orthogonal set. % for $\mathcal{W}_0 =\mathscr{V}_1\cap \mathscr{V}_0^\perp  \subseteq L^2(\Lambda^\infty, M)$. 
%{\color{purple}E: but the $f^{i, v}$ are not actually norm 1, are they? So this is just an orthogonal basis.}
Therefore, the functions $\{f^{i,v}\}_{i, v}$ span the subspace $\mathcal W_{0, \Lambda}\subseteq L^2(\Lambda^\infty, M)$ from \cite[Theorem 4.2]{FGKP}, which we will henceforth call $\mathcal W_0$.

The following Theorem, which was proved in \cite{FGKP}, justifies our labeling of the orthogonal decomposition \eqref{eq:wavelet-decomp} as a wavelet decomposition: the subspaces $\mathcal W_n$ are given by applying ``scaling and translation'' operators $S_\lambda$ to the finite family of ``mother functions'' $\{ f^{i,v}\}_{i,v}$.
\begin{thm}\cite[Theorem 4.2]{FGKP}
\label{thm:wavelets}
Let $\Lambda$ be a finite, strongly connected $k$-graph and define $\mathscr V_0 := \text{span} \{\chi_{[v]}: v \in \Lambda^0\}$. Let $\mathscr V_0 := \text{span} \{ \chi_{[v]}: v \in \Lambda^0\}$, and 
set
 	\[\mathcal W_n = \text{span}\{ S_\lambda f^{i, s(\lambda)}: d(\lambda) = (n, \ldots, n), 1 \leq i \leq \#(D_{s(\lambda)}) - 1\}\]
 	for each $n \in \N$. Then $\{ S_\lambda f^{i, s(\lambda)}: d(\lambda) = (n, \ldots, n), 1 \leq i \leq \#(D_{s(\lambda)}) - 1\}$ is a basis for $\mathcal W_n$ and 
 	\begin{equation}
 	\label{eq:wavelet-decomp}
 	L^2(\Lambda^\infty, M) \cong  \mathscr V_0 \oplus \bigoplus_{n=0}^\infty \mathcal W_n.\end{equation}
\end{thm}

\section{Spectral triples of Consani-Marcolli type for strongly connected finite higher-rank graphs}
\label{sec-Consani-Marcolli-spectral-triples-for-k-graphs}

In Section 6 of \cite{consani-marcolli}, Consani and Marcolli construct a spectral triple for the Cuntz-Krieger algebra $\mathcal{O}_A$ associated to a matrix $A \in M_n(\N)$.
Recall from \cite{kprr} that if $E$ is the 1-graph with adjacency matrix $A$, then $\mathcal{O}_A \cong C^*(E)$.

In this section, we generalize the construction of Consani and Marcolli to build spectral triples for higher-rank graph $C^*$-algebras $C^*(\Lambda)$.
For these spectral triples (described in Theorem~\ref{thm-Consani-Marcolli-spectral-triples-k-graphs} below), it is shown in Theorem \ref{thm:CM-Dirac-wavelets} that the eigenspaces of the Dirac operator agree with the wavelet decomposition from \cite{FGKP}. We also discuss in Remark \ref{rmk:CM-rectangle-wavelets} at the end of the section how to modify the construction of the spectral triple to make the eigenspaces of the Dirac operator compatible with the $J$-shape wavelets of \cite{FGKP2}.

\begin{defn} Let $\Lambda$ be a finite, strongly connected $k$-graph.  Define $\mathcal{R}_{-1} \subset L^2(\Lambda^\infty, M)$ to be the linear subspace of constant functions on $\Lambda^\infty$. 
 For $s\in\N$, define  $\mathcal{R}_s \subset L^2(\Lambda^\infty, M)$ by
 \[
 \mathcal{R}_s=\text{span} \left\{ \chi_{[\eta]} : \ \eta \in \Lambda, \ \sup \{ d(\eta)_i: 1 \leq i \leq k\} \leq s \right\},
 \]
 where $d(\eta) = (d(\eta)_1, \ldots, d(\eta)_k) \in \N^k$.
 
 Let $\Xi_s$ be the orthogonal projection in $L^2(\Lambda^\infty, M)$ onto the subspace $\mathcal{R}_s$.
       For a pair $(s,r)\in \N\times ( \N \cup \{-1\}) $ with $s>r$, let
        \begin{equation*}\label{eq:ortho-proj}
        \widehat{\Xi}_{s,r} =\Xi_s  - \Xi_{r}.
        \end{equation*}
        Since $\mathcal{R}_r \subset \mathcal{R}_s$, $\widehat{\Xi}_{s,r}$ is the orthogonal projection onto the subspace $\mathcal{R}_s \cap ({\mathcal{R}_{r} })^{\perp}$.

      Given an increasing sequence $\alpha= \{ \alpha_q\}_{q\in \N}$  of positive real numbers with $\lim_{q \to \infty} \alpha_q = \infty $, we define an operator $D$ on $L^2(\Lambda^\infty, M)$ by
      
\begin{equation} \label{eq:Dirac}        
D:=\sum_{q\in \N} \alpha_q\;\widehat{\Xi}_{q, q-1}.
\end{equation}

\end{defn}

 Note first that the operator $D$ has eigenvalues $\alpha_q$ with eigenspaces $\mathcal{R}_{q}\cap \mathcal{R}_{(q-1)}^\perp$ by construction. Also note that when $\Lambda$ has one vertex, $\mathcal{R}_{-1}=\mathcal{R}_0$ and the orthogonal projection $\widehat{\Xi}_{0,-1}$ is the zero projection.

\begin{prop}\label{pr:dirac-self-adjoint}
The operator $D$ on $L^2(\Lambda^\infty, M)$ of Equation \eqref{eq:Dirac}  is unbounded and self-adjoint.
  \end{prop}
 
 \begin{proof} 
 The fact that $D$ is unbounded follows from the hypothesis that $\lim_{q \to \infty} \alpha_q = \infty$. 
Thus, to see that $D$ is self-adjoint we must first check that it is densely defined, and then show that $D$ and  $D^*$ have the same domain.  For the first assertion, recall from Lemma 4.1 of \cite{FGKP} that 
\[ \{[\eta]: d(\eta) = (n, \ldots, n)\text{ for some }n \in \N\} \]
 generates the topology on $\Lambda^\infty$, and hence $
 \text{span} \{\chi_{[\eta]}:d(\eta)=(n,n,\dots,n),\; n\in \N\}
$
 is dense in $L^2(\Lambda^\infty, M)$.
  Given such a ``square'' cylinder set $[\eta]$ with $d(\eta) =(s, \ldots, s)$, since $\chi_{[\eta]} \in \mathcal{R}_s$, we can write $\chi_{[\eta]} = \sum_{r\leq s} \widehat{\Xi}_{r, r-1}(\chi_{[\eta]})$.  Then, 
\begin{equation*}
%\label{eq:S-lambda-cylinder}
D (\chi_{[\eta]}) = \sum_{r\leq s} \alpha_r \widehat{\Xi}_{r,r-1} (\chi_{[\eta]}) ,
\end{equation*}
which is a finite linear combination of vectors with finite $L^2$-norm, and hence is in $L^2(\Lambda^\infty, M)$. 
 In other words, for any finite linear combination $\xi$ of characteristic functions of square cylinder sets, $D\xi$ is in $L^2(\Lambda^\infty, M)$.  Thus $D$ is defined on (at least) the finite linear combinations of square cylinder sets, which form a dense subspace of $L^2(\Lambda^\infty, M)$.
 
Moreover, our definition of $D$ as a diagonal operator on $L^2(\Lambda^\infty, M)$ with real eigenvalues implies that  $D = D^*$ formally; since the operators $D$ and $D^*$ are given by the same diagonal formula, their domains also agree, and hence we do indeed have $D = D^*$ as unbounded operators.
 \end{proof}

\begin{prop} 
Let $D$ be the operator on $L^2(\Lambda^\infty, M)$ given in \eqref{eq:Dirac}.      For all complex numbers $\lambda \not\in \{ \alpha_n \}_{n \in \N}$, the resolvent $R_\lambda(D) := (D - \lambda)^{-1}$ is a compact operator on $L^2(\Lambda^\infty, M)$.
      \label{pr:cpt-resolvent}
    \end{prop}
\begin{proof}
By definition, $D$ is given by multiplication by $\alpha_q$ on $\mathcal{R}_{q}\cap \mathcal{R}_{(q-1)}^\perp$.  Consequently, for all $q 
\in \N$, $(D- \lambda)^{-1}$ is given by multiplication by $\frac{1}{\alpha_q - \lambda}$ on $\mathcal{R}_{q}\cap \mathcal{R}_{(q-1)}^\perp$.

Since $\lambda \not \in \{\alpha_n\}_{n\in \N}$ and $\lim_{n\to \infty} \alpha_n=\infty$, given $\epsilon > 0 ,$ we can choose $N$ so that for all $n \geq N$, $\frac{1}{|\alpha_n -\lambda |} < \epsilon$.  Fix $s\in \N$; then for any $f \in \mathcal{R}_s \cap \mathcal{R}_{s-1}^\perp$ of norm 1, 
\begin{align*}
\| \left( \sum_{q=1}^N \frac{1}{\alpha_q - \lambda}\widehat \Xi_{q, q-1} (f) \right)& - (D-\lambda)^{-1}(f)  \|  
= \| \sum_{q > N} \frac{1}{\alpha_q - \lambda} \widehat{\Xi}_{q, q-1}(f) \|  \\
&= \begin{cases}
\left| \frac{1}{\alpha_s - \lambda} \right| \, \|f\| & \text{ if } s > N \\
0 & \text{ if } s \leq N
\end{cases} \\
& < \epsilon,
\end{align*}
since $\|f \| = 1$ by hypothesis.  Since the subspaces $\{  \mathcal{R}_s \cap \mathcal{R}_{s-1}^\perp: s \in \N_0\}$ span $L^2(\Lambda^\infty, M)$,
it follows that $(D - \lambda)^{-1}$ is the norm limit of finite rank operators and hence is compact. 
\end{proof}

 \begin{thm} \label{thm-Consani-Marcolli-spectral-triples-k-graphs} 
 Let $\Lambda$ be a finite, strongly connected $k$-graph, and denote by $\pi$ the representation of $C^*(\Lambda)$ on $L^2(\Lambda^\infty, M)$ given in Equations \eqref{eq:S_lambda} and \eqref{eq:S-lambda-star}. Let $\mathcal{A}_\Lambda$ be the dense $\ast$-subalgebra of $C^*(\Lambda)$ given in Definition~\ref{def:dense-subalg} and let $D$ be the operator given in \eqref{eq:Dirac}.
 If {there exists a constant $C\ge 0$ such that} the sequence $\alpha= \{\alpha_q\}_{q \in \N}$ satisfies 
 \[
| \alpha_{q+1} - \alpha_q | \leq C,\ \forall q \in \N, 
 \] 
 then the commutator $[D,\pi(a)]$ is a bounded operator on $L^2(\Lambda^\infty, M)$ for any $a \in \mathcal{A}_\Lambda$.
 
 Combined with the above results, this implies that the data $(\mathcal{A}_\Lambda,  L^2(\Lambda^\infty, M), D)$ gives a spectral triple for $C^*(\Lambda)$. 
 \end{thm}
\begin{proof}
To prove that $(\mathcal{A}_\Lambda, L^2(\Lambda^\infty, M), D)$ is a spectral triple we need to show that $D$ is self-adjoint, $(D^2+1)^{-1}$ is compact and $[D,\pi(a)]$ is bounded for all $a \in \mathcal{A}_\Lambda$. The first statement is the content of Proposition \ref{pr:dirac-self-adjoint}, and the second follows from Proposition \ref{pr:cpt-resolvent}, thanks to the fact that $\pm i \not \in \{ \alpha_n\}_{n\in \N}$ and hence $(D \pm i )^{-1}$ is compact. 
Thus, to complete the proof of the Theorem, we will now show that $[D, \pi(a)]$ is bounded for all finite linear combinations $a = \sum_{i \in F } c_i s_{\lambda_i} s_{\eta_i}^*\in \mathcal{A}_\Lambda$, {where $c_i\in \C$.}

Given $\lambda \in \Lambda$, write $\max_\lambda = \max_j \{ d(\lambda)_j\}$   and $\min_\lambda = \min_j \{d(\lambda)_j\}$. Then the formula \eqref{eq:S_lambda} implies immediately that, for any fixed $s\in \N$,
 the operator  $S_\lambda$ on $L^2(\Lambda^\infty, M)$ takes $\mathcal{R}_s$ to $\mathcal{R}_{s+\max_\lambda}$.
      
Moreover, Equation \eqref{eq:S-lambda-star} implies that the operator $S_\lambda^*$ on $L^2(\Lambda^\infty, M)$ takes $\mathcal{R}_s$ to $\mathcal{R}_{s-\min_\lambda}$ if $\min_\lambda \leq s$, and to $\mathcal{R}_0$ otherwise.
To see this, suppose $\chi_{[\eta]} \in \mathcal{R}_s$ and $d(\eta) = (n_1, \ldots, n_k)$. 
Then $S_\lambda^* \chi_{[\eta]}$ is a linear combination of cylinder sets $\chi_{[\zeta]}$ with 
\[d(\zeta)_i = \begin{cases} 0, & d(\lambda)_i  \geq d(\eta)_i \\
d(\eta)_i - d(\lambda)_i, & d(\lambda)_i < d(\eta)_i
\end{cases}\]
Consequently, we see that (as desired) 
     \begin{align*}
     \max\{ d(\zeta)_i \} & = \max\{0,  n_i - d(\lambda)_i: 1 \leq i \leq k \}  \leq s - \textstyle{\min_\lambda}. \end{align*}
If $s<\min_\lambda$, then $n_i - d(\lambda)_i \leq 0$ for all $i$, so  $S^*_\lambda \chi_{[\eta]}\in \mathcal{R}_0$ for all $\chi_{[\eta]}\in \mathcal{R}_s$.
      
  Similarly,  if $f \in \mathcal{R}_s^\perp$, then $S_\lambda f \in \mathcal{R}_{s + \min_\lambda}^\perp$.  
   %   This follows from the fact that $\langle S_\lambda f, g \rangle = \langle f, S_\lambda^*g\rangle $ and the fact that $S_\lambda^*$ takes $\mathcal{R}_r$ to $\mathcal{R}_{r -\min_\lambda}$. 
   Namely, if  $\langle f, h \rangle = 0$ for all $h \in \mathcal{R}_s$, then our description of $S_\lambda^*$ above yields
      \[ \langle f, S_\lambda^* g\rangle = 0 \ \forall \ g \in \mathcal{R}_{s + \min_\lambda}.\]
An analogous argument shows that  $S_\lambda^*$  takes $\mathcal{R}_s^\perp$ to $\mathcal{R}_{s-\max_\lambda}^\perp$ if $s \geq \max_\lambda$. 
 
Now fix $q \in \N, \ f\in \mathcal{R}_q \cap \mathcal{R}_{q-1}^{\perp}$,  and fix $\lambda, \mu \in \Lambda$ with $s(\lambda) = s(\mu)$. 
We use  the reasoning of the previous paragraphs to identify the subspaces $\mathcal{R}_s, \mathcal{R}_t^\perp$ which contain  $S_\lambda S_\mu^* f$.

  If $ \max_\mu \geq q$, then we cannot guarantee that $S_\mu^* f$ is orthogonal to any $\mathcal{R}_t$ with $t \geq 0$; in order to do so, we must have $\langle S_\mu^* f, \xi \rangle = \langle f, S_\mu \xi \rangle = 0$ for all $\xi \in \mathcal R_t$.  In other words, we must have $S_\mu \xi \in \mathcal R_{q-1}$ for all $\xi \in \mathcal R_t$.  However, $S_\mu $ takes $ R_t $ into $ \mathcal R_{t + \max_\mu} \supsetneqq \mathcal R_{q-1}$ if $\max_\mu \geq q$ and   $t \geq 0$.
  
Moreover, if $q < \min_\mu$, then $S_\mu^* f\in \mathcal{R}_0$.  Thus,  
  \[ q < \textstyle{\min_\mu} \Rightarrow S_\lambda S_\mu^* f \in \mathcal{R}_{\max_\lambda}; \quad \textstyle{\min_\mu} \leq q \leq  \max_\mu \Rightarrow S_\lambda S_\mu^* f \in\mathcal{R}_{q + \max_\lambda - \min_\mu}; \]
  \[ q > \textstyle{\max_\mu} \Rightarrow S_\lambda S_\mu^* f \in \mathcal{R}_{q + \max_\lambda - \min_\mu} \cap \mathcal{R}_{(q-1) + \min_\lambda - \max_\mu}^\perp.\]

For now, assume $q > \max_\mu$.  Writing $g= S_\lambda S_\mu^*  f$,
  we have 
 \begin{align*} 
 g & = \Big( \Xi_{q+\max_\lambda - \min_\mu }-\Xi_{(q-1)+\min_\lambda - \max_\mu} \Big) g =\sum_{w=q+\min_\lambda - \max_\mu }^{q+\max_\lambda - \min_\mu }  \Big( \Xi_{w}-\Xi_{w-1} \Big) g
\end{align*}
and consequently 
  $$ D(S_\lambda S_\mu^* f) =: D g = \sum_{w=q+\min_\lambda - \max_\mu }^{q+\max_\lambda - \min_\mu } D \Big( \Big( \Xi_{w}-\Xi_{w-1} \Big) g \Big) = \sum_{w=q+\min_\lambda-\max_\mu}^{q+\max_\lambda-\min_\mu} \, \alpha_w\,  \Big( \Big( \Xi_{w}-\Xi_{w-1} \Big) g \Big). $$ 
 It now follows that, if $f \in \mathcal R_q \cap \mathcal R_{q-1}^\perp$ for $q > \max_\mu$,
    \[
   [D, S_\lambda S_\mu^* ] f= DS_\lambda S_\mu^* f - S_\lambda S_\mu^* D f =  \sum_{w=q+\min_\lambda - \max_\mu }^{q+\max_\lambda - \min_\mu }  \, ( \alpha_w - \alpha_q) \Big( \Big( \Xi_{w}-\Xi_{w-1} \Big) S_\lambda S_\mu^* f \Big)  .
    \]
 Consequently,  since $|\alpha_w - \alpha_{w-1}| \leq C$ for all $w$, 
    \[
    \begin{split}
      \Vert  [D, S_\lambda S_\mu^* ] f \Vert &\leq   \sum_{w=q+\min_\lambda - \max_\mu }^{q+\max_\lambda - \min_\mu }  \, |\alpha_w-\alpha_q|\,  \Vert  S_\lambda S_\mu^* f \Vert  \\ 
        &\leq \|S_\lambda S_\mu^* f\|  \sum_{w= q+ \min_\lambda - \max_\mu}^{q + \max_\lambda - \min_\mu} C |w-q| = \|S_\lambda S_\mu^*  f\|  C \sum_{t = \min_\lambda - \max_\mu }^{\max_\lambda - \min_\mu } | t|  .
     \end{split}
        \]
  Since $S_\lambda S_\mu^*$ is a partial isometry and hence norm-preserving,  
  whenever $f \in \mathcal R_q \cap \mathcal R_{q-1}^\perp$ for $q > \max_\mu$,  $\| [D, S_\lambda S_\mu^*] f\|$ is bounded above by a constant which depends only on $\lambda$ and $\mu$.

  If we have $\min_\mu \leq q \leq \max_\mu$, since we no longer know that $S_\lambda S_\mu^* f \in \mathcal{R}_t^\perp$ for any $t$, in calculating $\| [D, S_\lambda S_\mu^* f\ \|$ we have to begin our summation over $w$ at zero, rather than at $q +  \min_\lambda - \max_\mu$. 
  In this case, the final (in)equality above becomes 
  \[ \| [D, S_\lambda S_\mu^*] f \| \leq \sum_{t = 1}^{\max_\lambda - \min_\mu} C t \|S_\lambda S_\mu^* f\|  + \sum_{t=1}^q  C t \| S_\lambda S_\mu^* f \|.\]
 In this case, $q \leq \max_\mu$, so we obtain the norm bound 
  \begin{align*} \| [D, S_\lambda S_\mu^*] f \| & \leq \| S_\lambda S_\mu^* f\| C \left( \frac{(\max_\lambda - \min_\mu)(\max_\lambda - \min_\mu +1)}{2} + \frac{\max_\mu (\max_\mu +1)}{2} \right) .
  \end{align*}
  In other words, $\|[D, S_\lambda S_\mu^*] f\|$ is again bounded by a constant which only depends on $\lambda$ and $\mu$.  A similar argument shows that if $q < \min_\mu$, $\Vert [D,S_\lambda S^*_\mu] f\Vert$ is bounded by a constant which only depends on $\lambda$ and $\mu$.  Since $\{\mathcal R_q \cap \mathcal R_{q-1} \}_{q\in \N}$ densely spans $L^2(\Lambda^\infty, M)$, it follows that $[D, S_\lambda S_\mu^* ]$ is a bounded operator for all $(\lambda, \mu) \in \Lambda \times \Lambda$ with $s(\lambda) = s(\mu)$.
   
By linearity, it follows  that $[D, \pi(a)]$ is bounded for all finite linear combinations $a = \sum_{i \in F} c_i s_{\lambda_i} s_{\eta_i}^*$ of the generators $s_\lambda s_\eta^*$ of $C^*(\Lambda)$. Since every element of the dense $*$-subalgebra $\mathcal{A}_\Lambda$ of $C^*(\Lambda)$ is given by such a finite linear combination, it follows that $(\mathcal{A}_\Lambda, L^2(\Lambda^\infty, M), D)$ is a spectral triple, as claimed.
 \end{proof}

 \begin{thm}
 \label{thm:CM-Dirac-wavelets}
 Let $(\mathcal{A}_\Lambda, L^2(\Lambda^\infty, M), D)$ be the spectral triple described in Theorem~\ref{thm-Consani-Marcolli-spectral-triples-k-graphs}.  
  The eigenspaces of the Dirac operator $D$ given in \eqref{eq:Dirac} agree with the wavelet decomposition 
 \[L^2(\Lambda^\infty, M) = \mathscr{V}_0 \oplus \bigoplus_{q=0}^\infty \mathcal{W}_q\]
 of Theorem \ref{thm:wavelets} above (also see \cite[Theorem 4.2]{FGKP}). In particular, 
 \[ \mathscr{V}_0 = \mathcal{R}_0 \supseteq \mathcal{R}_{-1}\;\; \text{  and  }\;\;
  \mathcal{W}_q = \mathcal{R}_{q+1}  \cap \mathcal{R}_{q}^\perp,\ q\geq 0.\]

 \end{thm}

 \begin{proof}
 By definition, $\mathcal R_{-1} \subseteq \mathcal{R}_0 =  \mathscr{V}_0 = \text{span}\{ \chi_{[v]}: v \in \Lambda^0\}$. 
 For the second assertion, 
recall that $\mathcal W_n = \text{span}\{ S_\lambda f: f \in \mathcal{W}_0, \ d(\lambda) = (q, q, \ldots, q)\}$.  Since $\max_\lambda = \min_\lambda = q$ for all such $\lambda$, each such $S_\lambda$ takes $\mathcal{R}_s \cap \mathcal{R}_{s-1}^\perp$ to $\mathcal{R}_{s+q} \cap \mathcal{R}_{s+q-1}^\perp$.   Thus, it suffices to see that $\mathcal W_0 \subseteq \mathcal R_1 \cap \mathcal R_0^\perp$, and that $\mathcal W_q$ and $\mathcal R_q \cap \mathcal R_{q-1}^\perp$ have the same dimension for all $q \in \N$.

For the first statement, recall that $\mathcal{W}_0$ was constructed precisely to be the span of a family $\{f^{i,v}\}$ of functions (see Equation \eqref{eq:f_iv}) which were orthogonal to $\mathscr{V}_0 = \mathcal{R}_0$.  Moreover, every function $f^{i,v}$ is a linear combination of characteristic functions $\chi_\eta$ with $d(\eta) = (1, \ldots, 1)$, and therefore lies in $\mathcal R_1 \cap \mathcal R_0^\perp$. 

 From the fact that $\{ S_\lambda f^{i, s(\lambda)}:  d(\lambda) = (q, q, \ldots, q), \ 1 \leq i \leq \#(D_{s(\lambda)}) -1\}$ is a basis for $\mathcal W_q$, and the factorization rule in $\Lambda$, it follows that   $\mathcal{W}_q$ has dimension 
 \[
 {\sum_{v \in \Lambda^0} \#( \Lambda^{(q,\ldots, q)}v) \cdot \left( \#( v \Lambda^{(1, \ldots, 1)})-1 \right)  = \#(\Lambda^{(q+1, \ldots, q+1)}) - \#(\Lambda^{(q, \ldots, q)}). }
 \]
 Moreover, we know from \cite[Lemma 4.1]{FGKP} that ``square'' cylinder sets generate the topology on $\Lambda^\infty$; it follows that $\mathcal{R}_s$ is  spanned by $\{ \chi_{[\lambda]}: d(\lambda) = (s, \ldots, s) \}$. Indeed, this set forms a basis for $\mathcal R_s$: if $d(\lambda) = d(\mu) = (s, \ldots, s)$, then the factorization rule implies that
 \[ \langle \chi_{[\lambda]}, \chi_{[\mu]} \rangle = \int_{\Lambda^\infty} \chi_{[\lambda]} \chi_{[\mu]} \, dM = \delta_{\lambda, \mu} M([\lambda]).\]
  Consequently, $\mathcal R_{q+1} \cap \mathcal R_q^\perp$ also has dimension $\#( \Lambda^{(q+1, \ldots, q+1)} ) - \#(\Lambda^{(q,\ldots, q)}).$  Hence, $\mathcal W_q = \mathcal R_{q+1} \cap \mathcal R_q^\perp$ for all $q \in \N$, as desired.
 \end{proof}
 
 \begin{rmk} \label{rmk:CM-rectangle-wavelets}
 Fix $J \in \N^k$ with $J_i > 0 $ for all $i$.  We described in Section 5 of \cite{FGKP2} how to construct wavelets with ``fundamental domain'' $J$ -- the original construction in Section 4 of \cite{FGKP} used $J=(1, \ldots, 1)$.  By defining 
 \[ \widetilde{\mathcal{R}}_s = \text{span} \{ \chi_{[\eta]}: d(\eta) \leq sJ\}\]
 we can construct a Dirac operator $\widetilde{D}$ on $L^2(\Lambda^\infty, M)$ which gives rise to a spectral triple $(\mathcal{A}_\Lambda, L^2(\Lambda^\infty, M), \widetilde D)$ whose eigenspaces agree with the wavelet decomposition given in Theorem 5.2 of \cite{FGKP2}.  We omit the details here as they are completely analogous to the proofs of Theorems \ref{thm-Consani-Marcolli-spectral-triples-k-graphs} and \ref{thm:CM-Dirac-wavelets} above.
 \end{rmk}

\bibliographystyle{amsplain}
\bibliography{eagbib}

\end{document}